\documentclass{article}
\usepackage[utf8]{inputenc}
\usepackage[T1]{fontenc}

\usepackage[english]{babel}

\usepackage[letterpaper,top=2cm,bottom=2cm,left=3cm,right=3cm,marginparwidth=1.75cm]{geometry}

\usepackage{amsmath}
\usepackage{amssymb}
\usepackage{amsthm}
\usepackage{graphicx}
\usepackage[colorlinks=true, allcolors=blue]{hyperref}
\usepackage{xfrac}
\usepackage{tikz}
\usetikzlibrary{arrows.meta}
\usetikzlibrary{math}
\usetikzlibrary{positioning}
\usepackage{array}
\usepackage{pgfplots}
\pgfplotsset{compat=1.18}

\usepackage{makecell}
\usepackage{multirow}

\usepackage{thmtools} 

\usepackage{stmaryrd}
\DeclareUnicodeCharacter{03B1}{\ensuremath{\alpha}}
\DeclareUnicodeCharacter{2200}{\ensuremath{\forall}}
\DeclareUnicodeCharacter{2983}{\ensuremath{\llbracket}}
\DeclareUnicodeCharacter{2081}{\ensuremath{_{\!1}}}
\DeclareUnicodeCharacter{2082}{\ensuremath{_{\!2}}}
\DeclareUnicodeCharacter{2984}{\ensuremath{\rrbracket}}
\DeclareUnicodeCharacter{2208}{\ensuremath{\in}}
\DeclareUnicodeCharacter{2227}{\ensuremath{\land}}
\DeclareUnicodeCharacter{2228}{\ensuremath{\lor}}
\DeclareUnicodeCharacter{2124}{\ensuremath{\mathbb{Z}}}
\DeclareUnicodeCharacter{2115}{\ensuremath{\mathbb{N}}}
\DeclareUnicodeCharacter{2260}{\ensuremath{\neq}}
\DeclareUnicodeCharacter{2203}{\ensuremath{\exists}}
\DeclareUnicodeCharacter{2286}{\ensuremath{\subseteq}}

\usepackage{listings}
\definecolor{keywordcolor}{rgb}{0.7,0.1,0.1}
\definecolor{tacticcolor}{rgb}{0.0,0.1,0.6}
\definecolor{commentcolor}{rgb}{0.4,0.4,0.4}
\definecolor{symbolcolor}{rgb}{0.0,0.1,0.6}
\definecolor{sortcolor}{rgb}{0.1,0.5,0.1}
\definecolor{attributecolor}{rgb}{0.7,0.1,0.1}
\usepackage{xcolor}
\usepackage{wasysym}
\newcommand\doublecheck{\textcolor{blue}{\checked\kern-0.6em\checked}}
\makeatletter
\newcommand\markerlessfootnote[1]{%
  \begingroup
    \renewcommand\thefootnote{}
    \footnotetext{#1}%
  \endgroup
}
\makeatother

\usepackage{multirow}
\usepackage{booktabs}

\allowdisplaybreaks

\newcommand{\newvtheorem}[2]{\newtheorem{#1}[theorem]{\llap{\textnormal{\doublecheck} }#2}}

\theoremstyle{plain}
\newtheorem{theorem}{Theorem}
\newvtheorem{vtheorem}{Theorem}

\newvtheorem{vlemma}{Lemma}
\newtheorem{claim}[theorem]{Claim}

\newvtheorem{vproposition}{Proposition}
\newtheorem{conjecture}[theorem]{Conjecture}
\newtheorem{problem}[theorem]{Problem}

\newvtheorem{vcorollary}{Corollary}
\newvtheorem{vobservation}{Observation}

\theoremstyle{definition}
\newtheorem{definition}[theorem]{Definition}

\newtheorem{construction}[theorem]{Construction}
\theoremstyle{remark}
\newtheorem*{remark}{Remark}

\newcommand{\Erdos}{Erd\H{o}s}
\newcommand{\Turan}{Tur\'{a}n}
\newcommand{\Kruckeberg}{Kr\"{u}ckeberg}

\newcommand{\bbN}{\mathbb{N}}
\newcommand{\bbZ}{\mathbb{Z}}
\newcommand{\abs}[1]{\lvert#1\rvert}

\title{
Forbidden Sidon subsets of perfect difference sets,\\
featuring a human-assisted proof
}
\author{
Boris Alexeev
\thanks{Please be aware that in other versions of this paper, ChatGPT and Lean are listed as authors.
However, arXiv policy is that ``generative AI language tools should not be listed as an author''.  Nota bene: ChatGPT did not write any of the text of this paper.}
\and
Dustin G.\ Mixon\thanks{Department of Mathematics, The Ohio State University, Columbus, OH} \thanks{Translational Data Analytics Institute, The Ohio State University, Columbus, OH}
}
\date{}

\begin{document}
\maketitle

\begin{abstract}
We resolve a \$1000 \Erdos{} prize problem, complete with formal verification generated by a large language model.

In over a dozen papers, beginning in 1976 and spanning two decades, Paul \Erdos{} repeatedly posed one of his ``favourite'' conjectures: every finite Sidon set can be extended to a finite perfect difference set.
We establish that $\{1,2,4,8,13\}$ is a counterexample to this conjecture.

During the preparation of this paper, we discovered that although this problem was presumed to be open for half a century, Marshall Hall,~Jr.\ published a different counterexample three decades \emph{before} \Erdos{} first posed the problem.
With a healthy skepticism of this apparent oversight, and out of an abundance of caution, we used ChatGPT to vibe code a Lean proof of both Hall's and our counterexamples.
\end{abstract}

\section{Introduction} \label{section:intro}

We note that this paper is written mostly in the style of an ordinary mathematics paper, so we suggest skipping ahead to Section~\ref{section:ai} if the reader is primarily curious about the role of artificial intelligence in this paper, including the sense in which this features a ``human-assisted proof''.\\[0pt]

Paul \Erdos{} wrote many papers on the subject of \emph{Sidon sets/sequences}, also called \emph{$B_2[1]$ (or~$B_2$ for short) sets/sequences}.
We give a definition and an equivalent restatement:
\begin{definition} \label{def:sidon}
A set~$A$ (typically of integers) is a \emph{Sidon set} if all differences $a-a'$ of distinct $a,a'\in A$ are distinct.
\end{definition}
\begin{remark}
Beware that harmonic analysts use the term ``Sidon set'' to mean something entirely different.
\end{remark}
\begin{vobservation} \label{obs:plusminus}
\markerlessfootnote{\hspace{-1ex}\doublecheck{} denotes that the corresponding result has been verified in Lean.}
Equivalently, $A$~is a Sidon set if all sums $a+a'$ with $a,a'\in A$ are distinct, up to re-ordering of the terms $a+a'=a'+a$.
\end{vobservation}
\begin{proof}
The equivalence follows from the fact that $a-b=c-d$ is equivalent to $a+d=b+c$, though note that in the case of differences we ignore all differences $a-a=0$, whereas in the case of sums we \emph{do} consider the sums $a+a$.
For example, $\{1,2,3\}$ is not a Sidon set because of the identity $2-1=3-2$ using differences or the identity $1+3=2+2$ using sums.
\end{proof}

\Erdos{} posed many problems about Sidon sets, especially regarding their sizes, many of which are still open.
For example, \Erdos{} asked how quickly the elements grow in ``the greedy Sidon set'', also known as the ``Mian--Chowla sequence''~\cite{MianChowla} (OEIS sequence A005282~\cite{A005282}):
\begin{problem}
\label{problem:340}
Let $A=\{1,2,4,8,13,21,31,45,66,81,97,\dotsc\}$ be the greedy Sidon sequence, constructed at each step by picking the smallest positive integer that preserves the Sidon property.  What is the order of growth of~$A$?
\end{problem}

\Erdos{} did not believe that the Mian--Chowla sequence was optimally dense.
Instead, he made a specific conjecture about the maximum asymptotic size of a Sidon set of natural numbers:
\begin{conjecture}
\label{conjecture:329}
There exists a Sidon set of natural numbers $A\subset\bbN$ so that
\[ \limsup_{n\to\infty} \frac{\abs{A \cap \{1,\dotsc,n\}}}{\sqrt{n}}=1. \]
\end{conjecture}
\Erdos{} and \Turan{}~\cite{ErTu41} proved that this $\limsup$ is at most $1$ for every Sidon set~$A$ by analyzing the finite case.
Meanwhile, \Erdos{} showed that there exists a Sidon set~$A$ achieving~$1/2$, and \Kruckeberg{}~\cite{Kr61} achieved~$1/\sqrt{2}$.
\Erdos{} saw one potential path to proving Conjecture~\ref{conjecture:329} via perfect difference sets.
\begin{figure}[t]
\centering
\begin{tikzpicture}
  \def\radius{2.5cm}
  \def\nticks{21}
  \draw[thick] (0,0) circle (\radius);
  \foreach \k in {0,3,4,6,7,8,9,10,11,12,13,14,16,18,19,20} {
    \pgfmathsetmacro{\angle}{90-360*\k/\nticks}
    \pgfmathsetmacro{\angleb}{-360*\k/\nticks}
    \path ({\angle}:\radius)           coordinate (tick\k outer);
    \path ({\angle}:{0.9*\radius})    coordinate (tick\k inner);
    \path ({\angle}:{1.2*\radius})     coordinate (label\k);
    \draw (tick\k outer) -- (tick\k inner);
    \node[font=\scriptsize, rotate=\angleb] at (label\k) {\k};
  }
  \foreach \k in {1,2,5,15,17} {
    \pgfmathsetmacro{\angle}{90-360*\k/\nticks}
    \pgfmathsetmacro{\angleb}{-360*\k/\nticks}
    \path ({\angle}:{1.1*\radius}) coordinate (bigTickOuter\k);
    \path ({\angle}:{0.9*\radius}) coordinate (bigTickInner\k);
    \path ({\angle}:{1.225*\radius})     coordinate (label\k);
    \draw[very thick] (bigTickInner\k) -- (bigTickOuter\k);
    \node[font=\scriptsize, rotate=\angleb] at (label\k) {\large \textbf{\k}};
  }
\end{tikzpicture}
\caption{An illustration of the perfect difference set $B=\{1,2,5,15,17\}\bmod{21}$.
One can check that every difference from~$1$ to~$20$ appears exactly once between the bold ticks; for example, $6$~is witnessed by $2-17\pmod{21}$.
Suppose you have a favorite set, such as $A=\{1,5,15\}$.
\Erdos{}'s conjecture asks whether or not it is always possible to find a perfect difference set~$B$ modulo \emph{some}~$v$ that contains your favorite set~$A$, assuming of course that $A$~doesn't already have a repeated difference.
In this paper, we disprove this conjecture: you will be disappointed if your favorite set is~$\{1,2,4,8,13\}$.}
\end{figure}
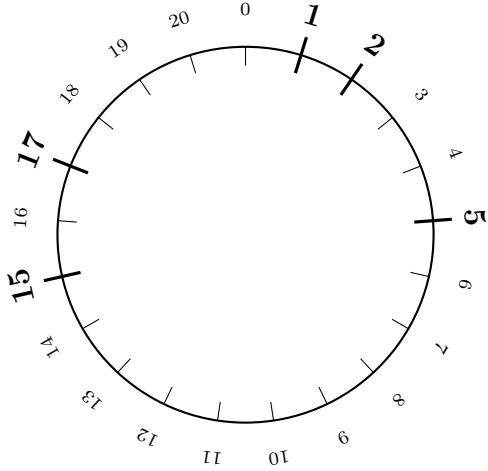
\begin{definition}
Given an abelian group~$G$, a set~$B\subset G$ is a \emph{perfect difference set} for~$G$ if the differences $d=b-b'$ of distinct $b,b'\in B$ represent every nonzero element $d\in G$ exactly once.
If $B$~is a set of integers, then we say that $B$ represents a perfect difference set modulo~$v>0$ if the images of~$B$ modulo~$v$ form a perfect difference set in the cyclic group $\bbZ/v\bbZ$ (the integers modulo~$v$).
In the sequel, we say \emph{finite perfect difference set} when the group is $\bbZ/v\bbZ$, and we say \emph{infinite perfect difference set} when the group is $\bbZ$.
\end{definition}
\begin{remark}
Note that if $B$~represents a perfect difference set modulo~$v$, we insist that the elements of~$B$ be distinct modulo~$v$, as otherwise $b_1\equiv b_2$ would induce a zero difference $b_1-b_2\equiv 0$ of distinct elements in~$B$.
Also, the name~$v$ is chosen for the modulus to align with standard practice in the theory of designs, where we also mention that the word ``perfect'' (the condition that each nonzero residue is represented exactly once) corresponds to $\lambda=1$.
\end{remark}
\begin{vobservation} \label{obs:pds_sidon}
Every perfect difference set is a Sidon set.
\end{vobservation}
By a simple counting argument, a perfect difference set modulo~$v$ can exist only if $v=q^2+q+1$ for some integer~$q$, in which case its size is $\abs{B}=q+1$.
Singer~\cite{Singer} proved that such a perfect difference set exists for each prime power~$q$.

\Erdos{} made the following optimistic\footnote{Besides optimism, \Erdos{} also mentioned as motivation similar results by Treash, Lindner, and others, for Steiner systems and other more complicated combinatorial structures.} conjecture, which would imply Conjecture~\ref{conjecture:329}:
\begin{conjecture}
\label{conjecture:707}
Every finite Sidon set can be extended to a finite perfect difference set.
\end{conjecture}
Indeed, one could then start with any Sidon set~$A$ and repeatedly perform two operations: arbitrarily extending~$A$ with some new element (which is always possible) and extending~$A$ to a perfect difference set (taking the smallest positive representatives modulo~$v$).
Since a perfect difference set up to~$q^2+q+1$ has size~$q+1$ (which is asymptotically equivalent to $\sqrt{q^2+q+1}$), Conjecture~\ref{conjecture:329} follows.

\Erdos{} posed Conjecture~\ref{conjecture:707} in at least 16~papers \cite{Er76b,Er77c,Er79g,Er80,Er81,Er84b,Er84d,Er85c,Er90Prob,Er91,Er92d,Er94c,Er97c,ErFr91b,ErGr80,WNT91} with publication dates from~1976 to~1997, differing slightly in the meaning of ``extended'' as well as which perfect difference sets are allowed.
This problem also appears in both sections~\textbf{C9} (``Packing sums of pairs'', including Sidon sequences) and~\textbf{C10} (``Modular difference sets and error correcting codes'', including perfect difference sets) of Guy's ``Unsolved Problems in Number Theory''~\cite{Guy04}.
In 1980~\cite{Er80}, \Erdos{} wrote:
\begin{quote}
[Conjecture~\ref{conjecture:329}] would follow from one of my favourite recent conjectures:
Let $1\le a_1<\cdots<a_k$ be a finite $B_2$ sequence.
Prove that it can be imbedded into a perfect difference set, i.e.\ there is a prime~$p$ and a set of $p+1$ residues $u_1,\ldots,u_{p+1}\bmod{p^2+p+1}$ so that all the differences $u_i-u_j$ are incongruent mod $p^2+p+1$ and the~$a$'s all occur amongst the~$u$'s.
I offer a thousand dollars for a proof or disproof of this conjecture.
\end{quote}
The main difference regarding the perfect difference sets is whether one requires~$p$ to be prime (as in the text above).
Sometimes there is no restriction, whereas other times there is an unqualified $p=q^\alpha$, thus presumably referring to a prime power.
At least once~\cite{Er91}, he wondered whether perhaps it might be true for all sufficiently large primes~$p$.

\Erdos{} famously offered prizes for solutions to many of his favorite problems.
The dollar amounts covered a wide range, ranging from ten dollars to thousands of dollars.
Only three problems are known~\cite{prizes} to have a prize larger than \$1000.
(One regarding prime gaps was solved by Maynard~\cite{Ma16} and Ford, Green, Konyagin, and Tao~\cite{FGKT16}.
The other two are still unsolved; both ask about the density of sets avoiding arithmetic progressions in different ways.
One is the Erdős conjecture on arithmetic progressions, which states that if $A$ is a set of positive integers such that $\sum_{n\in A} 1/n$ diverges, then $A$ contains arbitrarily long arithmetic progressions.)

If one requires $p$~to be prime (as in \Erdos{}'s statement of his \$1000 problem above), then we exhibit a particularly simple counterexample with four elements (the first powers of two):
\begin{restatable}{vtheorem}{theoremp}\label{1248}
The finite Sidon set $\{1,2,4,8\}$, which is the first \emph{four} terms of the Mian--Chowla sequence, does not extend to a finite perfect difference set modulo $v=p^2+p+1$ for any prime~$p$.
\end{restatable}
If one allows for an arbitrary perfect difference set (as in our statement of Conjecture~\ref{conjecture:707}), then we exhibit a counterexample with five elements:
\begin{restatable}{vtheorem}{theoremq}\label{theorem:main}
The finite Sidon set $\{1,2,4,8,13\}$, which is the first \emph{five} terms of the Mian--Chowla sequence, does not extend to a finite perfect difference set modulo any $v>0$.
\end{restatable}

During the preparation of this paper, we were surprised to learn that Marshall Hall,~Jr.\ already disproved Conjecture~\ref{conjecture:707} in~1947~\cite{Hall}, writing the following:
\begin{quote}
From this theorem it immediately follows that there are many sets of integers satisfying the conditions of [Definition~\ref{def:sidon}] which cannot be extended to any finite [perfect] difference set.
For example the set $-8, -6, 0, 1, 4$ may not be so extended.
\end{quote}
Because one can translate perfect different sets, an equivalent counterexample containing only positive integers can be produced by adding~$9$ to obtain $\{1,3,9,10,13\}$.

Clearly, it appears that \Erdos{} was not aware of this result.
The authors of this paper were also unaware of this result, even though they performed a reasonably deep literature search prior to starting this project.
(In fact, no large language model could find Hall's result, even with substantial prompting that the result indeed exists.) 
Instead, the paper was discovered by accident when searching for support for Conjecture~\ref{conjecture:desargues}.

We were completely taken aback by our discovery of Hall's result.
Hall's paper appeared in a famous journal, and it is clear that the paper is well-known, since, for example, it is cited by Guy~\cite{Guy04} just two sentences before stating Conjecture~\ref{conjecture:707} in section~\textbf{C10}.
So then why is Hall's counterexample not an accepted resolution of this conjecture?
Maybe there's a missing hypothesis?
Maybe it contains a deeper flaw?
In pursuit of the truth, we were determined to obtain a formal proof; we chose Lean~\cite{Lean:4} as our proof assistant because many other mathematicians have done the same recently.
However, since we're not experts in Lean (nor any other proof assistant), we prompted ChatGPT~\cite{OpenAI25} to write the code for us.
In the end, we obtained a human-assisted Lean proof of Hall's counterexample, and of Theorems~\ref{1248} and~\ref{theorem:main} (mind the double check marks \doublecheck{} above).

The resulting formal proof is thousands of lines, nearly all of which were written by ChatGPT.
Accordingly, we believe that ChatGPT is properly \emph{an author} of the formal proof accompanying this paper\footnote{However, the policy of the arXiv and other venues is that generative AI should not be listed as an author.}.
Unfortunately, large language models are known to hallucinate and otherwise produce incorrect results, so we would not be able to trust this proof unless it was in a formal language.
Therefore, we believe Lean is also an author, or perhaps ChatGPT and Lean may be credited together.
Even so, while our formal proofs were written and verified by a large language model and a proof assistant, respectively, we stress that this paper was written the old-fashioned way by human authors.

Before diving into the proof, we briefly discuss \emph{infinite} Sidon and perfect difference sets in Section~\ref{section:infinite}.
In Section~\ref{section:direct}, we give a direct proof of Theorem~\ref{1248}.
In Sections~\ref{section:plane} and~\ref{section:main}, we follow Hall's proof by building up some mathematical structure involving projective planes in order to prove Theorems~\ref{1248} and~\ref{theorem:main}, as well as the correctness of Hall's original counterexample.
Section~\ref{section:lean} discusses our Lean proof, which we vibe coded using ChatGPT, while Section~\ref{section:ai} discusses the broader role of AI in this paper.
We conclude in Section~\ref{section:future} with ideas for future directions.

\section*{Acknowledgments}

The online collection \url{erdosproblems.com} includes Problem~\ref{problem:340} as~$\#$340~\cite{dotcom340}, Conjecture~\ref{conjecture:329} as~$\#$329~\cite{dotcom329}, and Conjecture~\ref{conjecture:707} as~\#707~\cite{dotcom707}.
(Problem \#44~\cite{dotcom44} may also be of interest, as it is intermediate between~\#707 and~\#329.)
The authors thank Thomas Bloom for creating and curating this very useful resource.
We also appreciate the lively community that has arisen around it.

Our formalization would have been much more difficult without the proof assistant Lean.
The authors are especially grateful to the Mathlib community for their unified library of mathematics in Lean4, and to the Formal Conjectures authors for including this conjecture in their repository.
Even with the extensive Mathlib library, we found it surprisingly painful to successfully write proofs in Lean, so we also thank OpenAI for producing a large language model (namely, ChatGPT with GPT-5 -- but without ChatGPT Pro) that is capable of meeting a human mathematician halfway (or at least partway).

DGM was partially supported by NSF DMS 2220304.

\section{Some minor remarks on infinite sets}
\label{section:infinite}

The following table summarizes whether a finite/infinite Sidon set can always be extended to a finite/infinite perfect difference set:
\begin{center}
\renewcommand{\arraystretch}{1.5}
\begin{tabular}{cc|cc}
\multicolumn{2}{c|}{%
  \multirow{2}{*}{%
    \makecell[c]{Does a Sidon set always extend\\to a perfect difference set?}%
  }%
} & \multicolumn{2}{c}{\textbf{perfect difference set}} \\
\multicolumn{2}{c|}{} & \textit{finite} & \textit{infinite} \\
\hline
\multirow{2}{*}{\textbf{Sidon set}}
& \textit{finite}  & No (this paper) & Yes (see Claim~\ref{claim:v=0}) \\
& \textit{infinite} & No (pigeonhole) & No (see Reference~\cite{CiNa08}) \\
\end{tabular}
\renewcommand{\arraystretch}{1}
\end{center}

We briefly mention the infinite case.

\begin{claim}
\label{claim:v=0}
Every finite Sidon set can be extended to an infinite perfect difference set.
\end{claim}
\begin{proof}
This result and its proof are stated formally by Hall~\cite[Theorem 3.1]{Hall}, but the proof is essentially a ``just-do-it''~\cite{JDI1,JDI2} greedy construction wrapped around the observation that there are only finitely many obstructions to adding a missing difference~$d$ by inserting two elements~$x$ and $x+d$.
\end{proof}

Alternatively, one may ask whether an \emph{infinite} Sidon set can always be extended to an infinite perfect difference set.
Here, the answer is no, as pointed out by Cilleruelo and Nathanson~\cite{CiNa08}.
An easy counterexample is $A=\{2b\mid b\in B\}$, where $B$ is an infinite perfect difference set.
Indeed, $A$ is a Sidon set, but adding an even number to $A$ causes an immediate problem, while adding two odd numbers also results in a contradiction.
Thus, $A$ cannot be extended to a perfect difference set.

Finally, obviously an infinite Sidon set cannot be extended to a finite perfect difference set.

Having briefly discussed these infinite cases, the remainder of this paper will focus on \emph{finite} Sidon and perfect difference sets, as \Erdos{} had done in Conjecture~\ref{conjecture:707}.

\section{A direct solution to \Erdos{}'s \$1000 problem} \label{section:direct}

The next two sections build up various structures involving projective planes in order to prove Theorems~\ref{1248} and~\ref{theorem:main} and confirm Hall's original counterexample following his original proof.
In what follows, we present a short, direct proof of Theorem~\ref{1248} that bypasses all of this extra structure.
(To be explicit, we note that the proof in this section was not verified in our Lean code, because our interest in verification was driven by the desire to check \emph{Hall's} proof specifically.
The \emph{theorem} is still labeled as verified because it is proven in a different manner in the following two sections.)

\theoremp*
\begin{proof}
Suppose, for contradiction, that~$B$ represents a perfect difference set modulo $v=p^2+p+1$ that contains $\{1,2,4,8\}$.
We identify $B$~with the set of residues comprising its images modulo~$v$, and write an equals sign to denote modular equality.
(Note that $\{1,2,4\}$ is a perfect difference set modulo $7=2^2+2+1$, but this does not count as an extension of $\{1,2,4,8\}$ precisely because $1$~and $8$~are identical modulo~$7$.)

As \Erdos{} already observed in the quotation in Section~\ref{section:intro}, the perfect difference set $B$ has cardinality $p+1$ because there are $v-1=p^2+p=(p+1)\cdot p$ nonzero residues modulo~$v$ and $\abs{B}\cdot(\abs{B}-1)$ ordered pairs of distinct elements of~$B$.
Since $B$ contains at least four elements by assumption, we have $p\ge 3$, whence the prime~$p$ is odd and $\abs{B}=p+1$ is even.

We prove a small result inline that will be useful later.
Consider an arbitrary $a\notin B$.
Define the function $f_a\colon B\to B$ by $f_a(b)=c$ if $a-b= c-d$ in the unique representation of $a-b$ with $c,d\in B$.
(The residue $a-b$ is nonzero precisely because $a\notin B$.)
This function~$f_a$ is an involution, because if $f_a(b)=c$ then clearly $a-c= b-d$ is the corresponding representation of $a-c$ and thus $f_a(c)=b$.
Finally, note that if $f_a$~has a fixed point~$b$, i.e.,~$f_a(b)=b$, that implies that $a-b= b-d$ or equivalently $2b= a+d$.

As the main part of the proof, for each $x\in B$, we will find two elements $b_x,d_x\in B$ so that
$$2b_x= 2(x-1)+d_x,$$
where moreover $d_x$ is different for different~$x$.
If~$x=2$, then we choose $b_2=d_2=2$ as in $2\cdot 2=2\cdot(2-1)+2$.
Otherwise, we will use the construction from the previous paragraph with $a=2(x-1)$.
Observe that $a\notin B$ because otherwise $x-2=a-x$ would be two distinct representations of the nonzero residue $x-2$ as differences in~$B$.
Because $f_a$~is an involution on the set~$B$ with even cardinality, it has an even number of fixed points.
One such fixed point is~$x$ because $2(x-1)-x=x-2$, so it has another fixed point, say~$b_x\ne x$.
Together with its corresponding~$d_x$, we obtain our desired identity $2b_x= 2(x-1)+d_x$.

Suppose $x\ne x'$ but $d_x = d_{x'}$.
Then we may subtract $2b_x= 2(x-1)+d_x$ and $2b_{x'}= 2(x'-1)+d_{x'}$ to get $2(b_x-b_{x'})= 2(x-x')$ or equivalently (because the modulus~$v$ is odd) $b_x-b_{x'}= x-x'$.
These two representations must be identical, so $b_x=x$ and $b_{x'}=x'$.
But in choosing~$b_x$ and~$b_{x'}$, we insisted that $b_x\ne x$ and $b_{x'}\ne x'$ except possibly if $x=2$ and $x'=2$, a contradiction.

Finally, since we have a different~$d_x$ for each~$x$, the~$d_x$s are a permutation of the finite set~$B$ and there exists an~$x$ for which $d_x=4$.
Expanding the corresponding identity gives $2b_x= 2(x-1)+4$ or (again because $v$~is odd) $b_x= x+1$.
The only representation of $1$ as a difference $b_x-x$ in~$B$ is as $1=2-1$, so $x=1$ and thus $d_1=4$.
Similarly, there exists an~$x'$ for which $d_{x'}=8$.
Reasoning similarly, noting the only representation of~$3$ as $3=4-1$, we obtain $x'=1$ as well and thus $d_1=8$, a contradiction.
\end{proof}

\begin{remark}
The condition that $p$~is prime is necessary for this counterexample.
Indeed,
$$\{1,2,4,8,16,32,64,128,256\}\equiv\{1,2,4,8,16,32,64,55,37\} \pmod{73}$$
is a perfect difference set mod~$73=8^2+8+1$.
Accordingly, we need another element in our counterexample in Theorem~\ref{theorem:main}.
\end{remark}

\section{Cyclic projective planes}
\label{section:plane}

Throughout the more structured solution of the problem, we use the terminology of (finite) projective planes.
We will gradually build up progressively more structure, which will prove very useful.

\begin{definition}
A projective plane is an incidence structure between \emph{points} and \emph{lines} such that
\begin{itemize}
\item given any two distinct points, there is exactly one line incident with both of them,
\item given any two distinct lines, there is exactly one point incident with both of them, and
\item the plane is not degenerate.
\end{itemize}
The specific phrasing of non-degeneracy varies across different sources.
Often the condition is that there are four points, no three of which are collinear.
At the present moment, Mathlib has the condition
$$p_1 \notin l_2 \wedge p_1 \notin l_3 \wedge p_2 \notin l_1 \wedge p_2 \in l_2 \wedge p_2 \in l_3 \wedge p_3 \notin l_1 \wedge p_3 \in l_2 \wedge p_3 \notin l_3,$$
described as ``three points in general position''.
\end{definition}
\begin{vobservation} \label{obs:order}
If the projective plane is finite, then there exists a positive integer $q\ge 2$, called the \emph{order} of the plane, such that there are $q^2+q+1$ points (each on $q+1$ lines) and $q^2+q+1$ lines (each with $q+1$ points).
\end{vobservation}
\begin{proof}
Pleasantly, these basic results about projective planes are already known to Mathlib.
\end{proof}
We very quickly recap some well-known facts about projective planes:
Given a finite field of order~$q$, one can construct a projective plane of order~$q$.
Thus, there exist projective planes of order~$q$ for all prime powers~$q$.
It is unknown whether there is a projective plane of any other order.
Order~$12$ is the smallest for which it is unknown whether a projective plane exists.

Desargues's theorem states that ``Two triangles are in perspective axially if and only if they are in perspective centrally'' (but because it is inessential for the remainder of the paper, we do not elaborate on what this means).
A projective plane satisfying Desargues's theorem is called \emph{Desarguesian}, which includes all planes defined from a field (as mentioned above).
There exist non-Desarguesian projective planes (though the known finite examples are still of prime power order).

We mention all of this background because perfect difference sets correspond to \emph{cyclic} projective planes.
This is a projective plane with the additional structure of a regular action by a cyclic group (a kind of \emph{collineation}).
We do not prove the equivalence of these two notions here, but rather only the single direction we need:

\begin{vlemma}[one direction of Theorem~2.1 in Hall~\cite{Hall}, but likely implicitly known to Singer~\cite{Singer} and others]
\label{projective_plane}
Suppose $B$~is a perfect difference set modulo $v=q^2+q+1$ with $q\ge 2$.
Considering the residues~$x$ modulo~$v$ as points, the translates~$B+y$ (also modulo~$v$) of~$B$ as lines, and incidence as set membership (so point~$x$ is on line~$B+y$ if $x\in B+y$, or equivalently $x-y\in B$) produces a projective plane of order~$q$.
\end{vlemma}
\begin{proof}
The axioms of the projective plane follow straightforwardly from the unique representations of all nonzero residues modulo~$v$ as differences in~$B$.
For example, if $x-x'=b-b'$ is the unique representation of the nonzero residue $x-x'$ with $b,b'\in B$, then $B+(x-b)=B+(x'-b')$ is the unique line through the distinct points~$x$ and~$x'$.
\end{proof}

Singer~\cite{Singer} showed that the projective planes constructed from finite fields are cyclic, and thus perfect difference sets exist modulo $v=q^2+q+1$ when $q$~is a prime power.
The ``prime power conjecture'' states that there are no cyclic projective planes (and thus no perfect difference sets) for any other order.
Unlike the situation above without the extra ``cyclic'' condition, this has been verified at least up to two billion~\cite{BG04}.
Because projective planes constructed from fields are Desarguesian, a related conjecture is the following:
\begin{conjecture}
\label{conjecture:desargues}
Every finite cyclic projective plane (and thus perfect difference set) is Desarguesian.
\end{conjecture}
We mention this conjecture for a couple reasons.
First, if it were known to be true, we have constructed alternate proofs of Theorem~\ref{theorem:main} (which we do not describe here).
Intuitively, this is because Desargues's theorem is a fairly strong restriction and interacts with the perfect difference set condition.
Second, it is precisely researching the support for this conjecture that led the authors to find Hall's paper~\cite{Hall}, where he says
\begin{quote}
The properties found in [that paper's Section~4] make it highly plausible that every finite cyclic plane is Desarguesian.
\end{quote}
However, because Conjecture~\ref{conjecture:desargues} is open, we will not mention Desargues's theorem in the remainder of this work.

We note here an observation about Singer's construction that was useful to us during the exploratory phase of this work.
The usual exposition of this construction is algebraic and chooses a primitive root in the field $GF(q^3)$, seen as a field extension of $GF(q)$, and then reasons about the (cyclic!)\ multiplicative group of $GF(q^3)$.
This construction is very nice, but we found it very practical to work with a different one on the computer:
\begin{construction}[\cite{Ha38Iso}, according to~\cite{Hall}]\label{construction}
Given constants $a_1,a_2,a_3\in GF(q)$, define a sequence $(x_k)$ of elements of $GF(q)$ via the initial conditions $x_0=0$, $x_1=0$, $x_2=1$, and the third-order linear recurrence relation $x_k=a_1 x_{k-1}+a_2 x_{k-2}+a_3 x_{k-3}$ for $k\ge 3$.
Then consider the indices $k$ for which $x_k=0$.
For some values of the constants $a_1,a_2,a_3$, these indices are periodic modulo $q^2+q+1$ and form a perfect difference set.
\end{construction}
Of course, the astute reader will see the relationship between this construction and Singer's.
Moreover, it is possible to make the imprecise words ``for some values of the constants $a_1,a_2,a_3$'' exact and to classify which values work.
The reason we wanted to share this construction is that we found that randomly choosing $a_1,a_2,a_3$, then simply checking whether the construction succeeded (and retrying otherwise) was very effective in practice and required minimal coding.

\section{Polarities and the remainder of Hall's proof}
\label{section:main}

Hall's proof factors through a handful of geometric results from Baer~\cite{Ba45,Ba46} that identify intricate aspects of the well-known duality between points and lines in projective planes.
We start by introducing the notions of \emph{polarity} and \emph{absolute} points and lines.

\begin{definition}
A \emph{polarity} (an involutive duality/correlation) for a projective plane is a map $\pi$ that switches its points and lines, preserves incidence, and has order two (applying it twice gives the identity).
\end{definition}

\begin{vlemma}[Theorem~2.3 in Hall~\cite{Hall}]
\label{polarity}
The map~$\pi$ given by $x\rightleftharpoons B-x$ (swapping the point $x$ with the line $B-x$) is a polarity for the construction from Lemma~\ref{projective_plane}.
\end{vlemma}
\begin{proof}
The map~$\pi$ is clearly an involution that swaps points and lines, and it preserves incidence:
\[
\text{$x \in B+y$}
~~~\Longleftrightarrow~~~ \text{$x-y=(-y)-(-x)\in B$} 
~~~\Longleftrightarrow~~~ \text{$-y\in B-x$} 
~~~\Longleftrightarrow~~~ \text{$\pi(B+y) \in \pi(x)$}.
\qedhere
\]
\end{proof}

\begin{definition}
Fix a polarity~$\pi$ for a projective plane.
A point~$p$ is called \emph{absolute} (with respect to the polarity~$\pi$) if~$p$ lies on its polar line~$\pi(p)$.
Dually, a line~$\ell$ is called \emph{absolute} if~$\ell$ contains its polar~$\pi(\ell)$.
\end{definition}

\begin{vlemma}[Lemma~4.1 in Hall~\cite{Hall}]
\label{absolute_n+1}
A point~$x$ is absolute with respect to the polarity~$\pi$ from Lemma~\ref{polarity} if and only if $2x\in B$.
It follows that there are exactly $q+1$ absolute points.
\end{vlemma}
\begin{proof}
For the first part, we have
\[
\text{point~$x$ is absolute}
~~~\Longleftrightarrow~~~ \text{$x$~lies on the line $B-x$}
~~~\Longleftrightarrow~~~ \text{$2x=x-(-x)\in B$}.
\]
We know from earlier that $v=q^2+q+1$, so $v$~is odd and $2$~has a multiplicative inverse modulo~$v$.
Thus, for each $b\in B$, there is a unique residue~$x$ such that $2x=b$, and there are exactly $q+1$ absolute points, one corresponding to each element of~$B$.
\end{proof}

\begin{vproposition}[Lemma in Baer~\cite{Ba45}]
\label{absolute1}
For an arbitrary polarity for an arbitrary projective plane:
An absolute line contains one and only one absolute point.
\end{vproposition}
\begin{proof}
By definition, an absolute line~$\ell$ contains its polar~$\pi(\ell)$, which is also absolute by definition, so it contains at least one absolute point.

Suppose~$p$ is an absolute point lying on the absolute line~$\ell$.
By the polarity property, the point~$\pi(\ell)$ lies on the line~$\pi(p)$.
It follows that~$p$ and~$\pi(\ell)$ are two points that both lie on the lines~$\ell$ and~$\pi(p)$.
If $p$~did not equal~$\pi(\ell)$, we would have two distinct lines passing through two distinct points in a projective plane, a contradiction.
Thus~$\pi(\ell)$ is the unique absolute point lying on an absolute line~$\ell$.
\end{proof}

\begin{vproposition}[Lemma in Baer~\cite{Ba46}]
\label{not_absolute}
For an arbitrary polarity for an arbitrary projective plane:
If a line is not absolute, then the number of points on it which are not absolute is even.
\end{vproposition}
\begin{proof}
Suppose the line~$\ell$ is not absolute, so it does not contain its polar~$\pi(\ell)$.
If $p$~is a point on~$\ell$, then denote by~$p'$ the intersection of the lines~$\ell$ and~$\pi(p)$.
By the properties of a polarity, $\pi(p)$~is the unique line passing through~$\pi(\ell)$ and~$p'$.
It also follows that~$p$ is absolute if and only if $p=p'$, and that $p''=p$.
The points on~$\ell$ which are not absolute therefore occur in pairs, proving our contention.
\end{proof}

\begin{vproposition}[Theorem~2 in Baer~\cite{Ba46}]
\label{odd1}
For an arbitrary polarity for an arbitrary projective plane, of \textbf{odd} order~$q$:
A line is absolute if and only if it carries exactly one absolute point.
\end{vproposition}
\begin{proof}
The proof is immediate.
If $\ell$~is absolute, then it contains exactly one absolute point by Proposition~\ref{absolute1}.
Otherwise, $\ell$~contains $q+1$~points in total by Observation~\ref{obs:order}, which is even.
An even number of those are not absolute by Proposition~\ref{not_absolute}, so the remainder comprise an even number of absolute points.
In particular, there is not exactly one absolute point.
\end{proof}

\begin{vproposition}[one direction of Corollary~1 (of Theorem~5) in Baer~\cite{Ba46}]
\label{odd2}
For an arbitrary polarity for an arbitrary projective plane of \textbf{odd} order~$q$:
If the polarity has exactly $q+1$ absolute points, then there are no more than two absolute points on any line.
\end{vproposition}
\begin{proof}
Consider an absolute point~$p$, which by definition lies on its polar line~$\pi(p)$.
Consider any other line~$\ell\ne\pi(p)$ through~$p$.
By Proposition~\ref{odd1}, it must contain at least one more absolute point.
(If the line~$\ell$ were absolute, it would contain the absolute point~$\pi(\ell)\ne p$, which would suffice for the claim.
But that actually contradicts Proposition~\ref{absolute1}, so $\ell$~is not absolute.)

We have identified all $q+1$ absolute points: besides~$p$, there are the $q$~distinct absolute points on each of the $q$~other lines through~$p$.
(The points are distinct because different lines through~$p$ intersect only at~$p$.)
Hence every line through~$p$ contains at most one absolute point other than~$p$.

This result applied to~$\ell$ proves that~$\ell$ contains at most two absolute points in total.
\end{proof}

\theoremp*
\begin{proof}[Another proof (see Section~\ref{section:direct} for the first proof)]
Construct the projective plane from Lemma~\ref{projective_plane}, and fix the polarity from Lemma~\ref{polarity}.

From Lemma~\ref{absolute_n+1}, we see that $1$, $2$ and $4$ are absolute points that all lie on the single line $B+0$.
Thus, $p$~is not odd by Proposition~\ref{odd2}.
Unfortunately, the sole even prime~$p=2$ is too small to extend the set~$A$.
\end{proof}

\begin{vproposition}[Theorem~1 in Baer~\cite{Ba46}]
\label{even1}
For an arbitrary polarity for an arbitrary projective plane of \textbf{even} order~$q$:
Every line carries an odd number of absolute points.
\end{vproposition}
\begin{proof}
If $q$~is even, then every line contains $q+1$~points, which is odd.
By Proposition~\ref{not_absolute}, an even number of these are not absolute, leaving an odd number of absolute points as desired.
\end{proof}

\begin{vproposition}[The direction (i) implies (iii), from Corollary~2 (of Theorem~5) in Baer~\cite{Ba46}]
\label{even2}
For an arbitrary polarity for an arbitrary projective plane of \textbf{even} order~$q$:
If the polarity has exactly $q+1$ absolute points, then all absolute points lie on a line.
\end{vproposition}
\begin{proof}
Suppose the line~$\ell$ contains a point~$p$ which is not absolute.
Each of the $q+1$~lines through~$p$ contain an absolute point by Proposition~\ref{even1}, which are all distinct and different from~$p$.
(As a few proofs ago: the points are distinct because different lines through~$p$ intersect only at~$p$.)
This accounts for all $q+1$~absolute points, so each line through~$p$ contains exactly one absolute point.
In other words, if a line~$\ell$ contains a point which is not absolute, then it contains exactly one absolute point.

There are at least $q+1\ge 2$ absolute points, so if we consider a line through any two of them, all points on that line must be absolute.
Furthermore, again by counting, that accounts for all $q+1$~absolute points.
\end{proof}

\theoremq*
\begin{proof}
Construct the projective plane from Lemma~\ref{projective_plane}, and fix the polarity from Lemma~\ref{polarity}.
Note that this projective plane has some order~$q$, which of course is either even or odd.

As before, from Lemma~\ref{absolute_n+1}, we see that $1$, $2$ and $4$ are absolute points that all lie on the single line $B+0$.
Thus, $q$~is not odd by Proposition~\ref{odd2}.

But if $q$~is even, all points on the line $B+0$ are absolute by Lemma~\ref{even2}.
In particular, $8$~is absolute and thus $2\cdot 8=16\in B$ by Lemma~\ref{absolute_n+1}.

We obtain a contradiction because $16-13=4-1$ violates the defining property of a perfect difference set.
\end{proof}

\begin{vtheorem}[the paragraph following Theorem~4.3 in Hall~\cite{Hall}]
\label{86014}
$\{-8,-6,0,1,4\}$ does not extend to a perfect difference set.
\end{vtheorem}
\begin{proof}
Construct the projective plane from Lemma~\ref{projective_plane}, and fix the polarity from Lemma~\ref{polarity}.

By Lemma~\ref{absolute_n+1}, we see that $-4$, $-3$, $0$, and $2$ are absolute points.
Of these, $-4$, $-3$, and $0$ lie on the line $B-4$.
However, $2$ does not lie on the line $B-4$ because otherwise $6\in B$ and the two equal differences $6-0=0-(-6)$ would violate the defining property of a perfect difference set.

This causes a problem for both odd and even~$q$.
If $q$~were odd, then too many absolute points lie on the same line, violating Proposition~\ref{odd2}.
If $q$~were even, then not all absolute points lie on the same line, violating Proposition~\ref{even2}.
\end{proof}

\begin{vcorollary} 
\label{139AD}
$\{1,3,9,10,13\}$ is a Sidon set that does not extend to a perfect difference set.
\end{vcorollary}
\begin{proof}
If $B$~is a perfect difference set and $c$ is an arbitrary constant, then $B+c$~is also a perfect difference set.
(The $c$s cancel when computing differences in~$B$.)

Therefore, $\{-8,-6,0,1,4\}$ extends to a perfect difference set if and only if $9+\{-8,-6,0,1,4\}=\{1,3,9,10,13\}$ does as well.
Hence we are done by Theorem~\ref{86014}.
\end{proof}

\section{Lean statement} \label{section:lean}

We used ChatGPT to vibe code\footnote{Vibe coding~\cite{wiki:vibe} refers to a style of programming where the user interacts with a large language model in order to generate code, which they do not verify except via the use of tools.
The term was popularized by Andrej Karpathy~\cite{Karpathy} in February 2025.
(The reference there also describes using voice interaction, which we did not use and does not seem to be an essential part of the accepted usage of the term now.)
As we discuss in Section~\ref{section:ai}, Lean is an ideal use case for vibe coding.} a proof.
The resulting proof consists of thousands of lines of spaghetti code, with many missteps and convoluted arguments.
Normally for programming code, this would be reason to distrust the code; however, Lean is a proof assistant that formally verifies arguments, and so we can be sure that the argument is correct even if its code is ugly.
As the Lean website~\cite{Lean:website} says:
``Lean’s minimal trusted kernel guarantees absolute correctness in mathematical proof, software and hardware verification.''
Even so, there is still one potential source of error: the statement that Lean verifies may not correspond properly to the statement that the human mathematician believes is being proven.

The Formal Conjectures~\cite{FormalConj} project is an initiative by Google DeepMind to create a large, open corpus of formalized statements of open conjectures (into Lean).
In particular, they are attempting to translate all of \Erdos{}'s problems, and we were lucky that this particular \Erdos{} problem had already been translated~\cite{FC707}.
Several variants were included.
The principal one included the condition that $v=q^2+q+1$ and $q$ is a prime power.
We were interested in proving the conjecture with no restriction on the modulus~$v$, which was included as well.
However, the corresponding statement in Lean was in fact incorrect for a subtle reason!
The statement in Lean conjectured that every finite Sidon set can be extended to a perfect difference set modulo~$v$, but it accidentally allowed $v=0$.
The case~$v=0$ means that $\bbZ/v\bbZ$ is simply $\bbZ$~itself, so this corresponds to the infinite case handled in Claim~\ref{claim:v=0}.
In particular, this case is both relatively easy and has the opposite resolution.

While our Lean \emph{proof} was generated entirely by ChatGPT, we did carefully check that our statement matched what we believed we were proving.
As an extra precaution, we proved several ``consistency checks'' that were not needed for the main proof, but would support the claim that the statement of the conjecture was translated correctly.
Specifically, this included Observations~\ref{obs:plusminus} and~\ref{obs:pds_sidon}.

The Lean proof itself is included as a supplementary file, but for the avoidance for doubt, we include a portion of the file here so that one can see the \emph{statement} being proven:

\begin{verbatim}
/- This proof has been verified on Lean Toolchain
version leanprover/lean4:v4.24.0 and Mathlib version
f897ebcf72cd16f89ab4577d0c826cd14afaafc7 (v4.24.0) -/

import Mathlib

/-- A Sidon set `A` is a set where all pairwise sums
`i + j` are unique, up to swapping the addends. -/
def IsSidon {α : Type*} [AddCommMonoid α] (A : Set α) : Prop :=
  ∀ ⦃i₁ i₂ j₁ j₂ : α⦄, i₁ ∈ A → i₂ ∈ A → j₁ ∈ A →
    j₂ ∈ A → i₁ + i₂ = j₁ + j₂ →
    (i₁ = j₁ ∧ i₂ = j₂) ∨ (i₁ = j₂ ∧ i₂ = j₁)

/-- `B` is a perfect difference set modulo `v` if there
is a bijection between non-zero residues mod `v` and
 distinct differences `a - b`, where `a, b ∈ B`. -/
def IsPerfectDifferenceSetModulo (B : Set ℤ) (v : ℕ) : Prop :=
  B.offDiag.BijOn (fun (a, b) => (a - b : ZMod v))
  {x : ZMod v | x ≠ 0}

-- (6216 lines omitted)

/--
**Erd\H{o}s problem 707**:
Any finite Sidon set of natural numbers can be embedded
in a perfect difference set modulo `v` for some `v ≠ 0`.
-/
def erdos_707 : Prop :=
  ∀ A : Set ℕ, A.Finite → IsSidon A →
    ∃ (B : Set ℤ) (v : ℕ),
      v ≠ 0 ∧
      (↑) '' A ⊆ B ∧
      IsPerfectDifferenceSetModulo B v

-- (101 lines omitted)

/--
The Sidon set {1, 2, 4, 8, 13} does not extend to a
perfect difference set modulo v for any nonnegative v.
-/
theorem not_erdos_707AM : ¬ erdos_707 :=
  not_erdos_707_given_counterexample
    counterexampleAM
    counterexampleAM_finite
    counterexampleAM_Sidon
    counterexampleAM_noExt

#print axioms not_erdos_707AM
-- 'not_erdos_707AM' depends on axioms:
-- [propext, Classical.choice, Quot.sound]
\end{verbatim}

We can verify that the statement of \Erdos{}'s conjecture indeed says that any finite Sidon set can be extended to a perfect difference set.
Because this Lean file compiles without any issues, it verifies that the theorem \verb`not_erdos_707AM` indeed has the type \verb`¬ erdos_707`.
This means we have proven the negation of \Erdos{}'s conjecture.
(One cannot ascertain from the snippet provided above specifically which counterexample is verified, as it is only mentioned in a comment.)

Our resulting proof is over 6000 lines (over a quarter of a megabyte) and consists of 26~definitions, 169~lemmas, and 4~theorems (the final verification of counterexamples).
On our ordinary laptop, the code takes slightly under half a minute to verify.\footnote{The laptop is ordinary, but we had to buy a new one in order to use all of the technology required by this project.}
These metrics provide some data in the discussion of the ``de~Bruijn factor''~\cite{Wied06}, which is defined as the ratio of the size of a formal proof over the size of the informal proof.
We will abstain from computing a specific number to represent this, but we wanted to note that unlike in de~Bruijn's original work, we found that the factor varied heavily depending on the details of the argument.
If an argument was straightforward and involved only some projective planes and polarities, then usually the factor was tiny.
If, however, an argument involved cardinalities and parities, it would explode.

As discussed more in the following section, the most egregious example was a basic claim about involutions that took us over $250$~lines of code, but which comprises a single sentence in this paper.

\section{Discussion of the uses of artificial intelligence} \label{section:ai}

Modern large language models, such as ChatGPT from OpenAI, Claude from Anthropic, and Gemini from Google DeepMind, have proven to be very useful in many aspects of mathematical research.

One popular use case of LLMs is literature search.
In particular, there have been several recent success stories on \url{erdosproblems.com} with using LLMs to locate solutions to \Erdos{} problems in the existing literature; see Tao~\cite{Tao:mathstodon} for a discussion.
Unfortunately, LLMs completely failed to locate Hall's solution to our \Erdos{} problem.
Why?
There seems to be a confluence of issues here:
\begin{itemize}
\item[1.]
Hall didn't state his result as a theorem, but rather as a throwaway sentence after a theorem.
(Of course, he couldn't predict that \Erdos{} would later make a huge fuss out of this problem, so he didn't know that a theorem statement would be necessary to signal the result's existence.)
\item[2.]
No human mathematician appeared to have noticed Hall's result.
As an egregious example, Guy~\cite{Guy04} cites Hall's paper two sentences before stating this \Erdos{} problem in section \textbf{C10}.
Despite being aware of (some of) the contents of Hall's paper, Guy failed to connect the dots.
As another perspective, it seems that the folks who solve \Erdos{} problems are generally good at analytic number theory or analytic aspects of combinatorics, but they are perhaps less familiar with the study of combinatorial designs.
\item[3.]
One might expect AI to help close this gap by reading the literature and responding to appropriately engineered prompts, but alas, it appears that the contents of Hall's paper might have been blocked from the AI's training set thanks to a paywall.
As evidence of this, we searched a verbatim passage from the first page of Hall's paper and got a Google hit, but a similar search using text from later pages fails.
In particular, we were very lucky that Hall's comments about Conjecture~\ref{conjecture:desargues} that we quote in Section~\ref{section:plane} appeared on the first page!
\end{itemize}

In various projects, LLMs are very helpful in writing exploratory code.
This is almost an ideal use case.
We began our project by generating many perfect difference sets to see which Sidon sets were evidently forbidden.
Accordingly, we could have used LLMs to help with this, but instead, the first author found it mildly amusing to discover and use Construction~\ref{construction}.
Through these investigations, we found that $\{1,2,4,8\}$ appeared to be forbidden for primes (and odd prime powers), and $\{1,2,4,8,13\}$ was forbidden even for powers of two.
These observations ultimately led to Theorems~\ref{1248} and~\ref{theorem:main}.

Sometimes, LLMs are helpful with brainstorming ideas, suggesting techniques, or even proving theorems.
In fact, various authors have reported LLMs identifying a crucial idea for a proof, supplying a lemma, or otherwise making it possible to complete a project that they were stuck on.
Sadly, for our problem, LLMs were not terribly helpful in this way.
Even after we knew what exactly to prove, it couldn't help us close the gap.

At some point during the course of this project, we eventually found Hall's paper and thus learned that the problem had been solved.
This was very confusing for us.
Was Hall's proof correct, or was it generally understood by the community to be flawed in some way?
It was clear to us that Lean would help us determine the truth, but we didn't know Lean, and it isn't terribly user-friendly.
However, ChatGPT can write Lean, so we decided to vibe code the whole proof.
It took a long time (about a week\footnote{We do not have a precise estimate of the time required, but ``a week'' is in the sense of ``working full-time, with substantial overtime''.  With respect to ordinary calendar time, it took a month.}), and the process was extremely annoying, but somehow it succeeded.

One thing we came to realize is that Lean is actually a perfect setting for vibe coding.
In general, vibe coding is good for generating code that appears to work well, but might contain some bugs.
As such, one must carefully test the resulting code before deploying it in important settings.
Modern applications of vibe coding include quick prototypes or fun side projects, i.e., settings in which bugs are not catastrophic.
It turns out that Lean is also a great match for vibe coding, but for a completely different reason: if the code runs, you can trust it!

People think of Lean proofs as \textit{computer-assisted} proofs, but this was entirely backwards from our experience.
What we experienced was much more like a \textit{human-assisted} (formal) proof.
An idealized version of our experience would look like this:
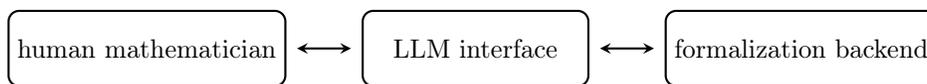
\begin{center}
\begin{tikzpicture}[
  node distance=1cm,
  box/.style={
    draw,
    rounded corners,
    minimum width=3cm,
    minimum height=1cm,
    align=center,
    thick
  },
  arrow/.style={
    <->,
    thick,
    >=stealth,
    shorten >=4pt,
    shorten <=4pt
  }
]
\node[box] (human) {human mathematician};
\node[box, right=of human] (llm) {LLM interface};
\node[box, right=of llm] (backend) {formalization backend};
\draw[arrow] (human.east) -- (llm.west);
\draw[arrow] (llm.east) -- (backend.west);
\end{tikzpicture}
\end{center}
In particular, the human mathematician engages in a productive conversation with an LLM, discussing a mathematical argument in natural language.
The LLM then translates and extrapolates the core ideas in Lean code, fighting with syntax and compile errors on its own, keeping the human away from this level of tedium.
When the LLM gets stuck on a fundamental logical step, it asks an appropriate followup question to the human mathematician.
After a few iterations, the LLM successfully formalizes the human mathematician's proof.

From this perspective, our very first interaction with ChatGPT was stunning.
We asked for it to state Proposition~\ref{absolute1} in Lean (meaning to state the proposition without proof), and it quickly replied not only with the statement but also with a correct proof.
In particular, we did not yet supply it with any argument, which we presumed would be necessary.

Sadly, most of the rest of our actual experience with human-assisted proof wasn't nearly this pleasant (reflective of ChatGPT's so-called ``jagged intelligence''), though we are hopeful for a future that is closer to the above idealization.
One of the early frustrations was while proving the next result, Proposition~\ref{not_absolute}.
It was promising at first, as ChatGPT itself constructed and proved the properties of the involution $p\mapsto p'$ without any hints at all.
However, it was completely stuck finishing the final step: that since the non-absolute points occur in pairs, there is an even number of them.

This is a curious inversion of the experience of describing the proof of Proposition~\ref{not_absolute} to another human mathematician:
The interesting and possibly ``difficult'' part to a human is the definition of~$p'$ from~$p$.
It may also be interesting to verify the desired properties of this map, but the final step (that if a finite number of items occur in pairs, then there are an even number of them) is so simple that depending on the verbosity of the writer, it may occupy only one or even no sentences in a writeup.

For us, however, convincing ChatGPT to formally prove the result that ``if $f$~is a fixed point--free involution on a finite set~$S$, then $S$~has even cardinality'' was a multi-day struggle!
Furthermore, the resulting argument is 250 lines of code, many of which deal with entirely trivial claims.
(This is not an inherent feature of Lean; in fact, almost certainly a succinct and nice proof exists and can be written by a competent human.
However, it was the best we could do vibe coding with our particular model.)

There were multiple recurring sources of difficulty with Lean and ChatGPT during the vibe coding process.
One was the multitude of notions of \emph{cardinality} in Mathlib, as applied to sets, finite sets, subtypes, etc.
Perhaps they all make sense when viewed appropriately, but ChatGPT was not helpful in translating between all of them.
Routinely we would seek to prove a basic lemma involving cardinality, only to be stymied because there were issues involving which specific cardinality was known or sought.

Another, very simple but frequent issue that arose involved parity.
Perhaps the ChatGPT model we used was trained heavily on mathlib3, the previous version of Mathlib, which has different definitions of basic concepts like even/odd and lemmas involving them.
We were often unsuccessful with prompting ChatGPT to look up the appropriate definitions and use them.
Thankfully this usually did not delay the work very much, as it's a very simple matter.
In general, it can be frustrating when different versions of a programming language are incompatible and the large language model of choice was trained on the ``wrong’’ version, relative to the use case.

Many of the issues we encountered would be alleviated greatly if the large language model interface we interacted with were integrated with Lean and appropriately trained for this interaction.
We believe that even a small amount of such fine tuning would have made our specific ``human-assisted proof'' a much more pleasant interaction.

Finally, when it came to writing the paper, we knew that ChatGPT would probably do a decent job, but we decided to do our part and write it the old-fashioned way.

\section{Future directions} \label{section:future}

In this paper, we identified forbidden Sidon sets of finite perfect difference sets.
There are a few opportunities for future work.

First, let $s$ denote the size of the smallest forbidden Sidon set.
Both our example and Hall's example establish that $s\leq 5$.
Meanwhile, the lower bound $s\geq 3$ follows from the fact that there are arbitrarily large finite perfect difference sets.
What is $s$?
Also, what are the forbidden Sidon sets of size $s$?
One might interpret this as a challenge to ``beat the AI'', \`{a}~la~\cite{Tao:MO}.

Can one find other \Erdos{} problems that were solved before they were posed?
Can AI help solve either of the two (ostensibly) open \Erdos{} problems that are worth more than \$1000?
Finally, how long until AI gets to the point where human-assisted proof is easier than conventional mathematics research?

\end{document}